\documentclass[10pt]{article}

\usepackage{comment}
\usepackage{amssymb,latexsym,amsmath,epsfig,amsthm} 
\usepackage{empheq}
\usepackage{here}
\usepackage{tabularx}
\usepackage{array}
\newtheorem{theorem}{Theorem}[section]
\newtheorem{lemma}[theorem]{Lemma}

\newtheorem{definition}[theorem]{Definition}
\newtheorem{example}[theorem]{Example}

\begin{document}

\Large{Maximum Nim and Josephus Problem}


\vspace{1.4cm}

\large{Shoei Takahashi,  Hikaru Manabe and Ryohei Miyadera}\\ \\

\ \ \ \ \ \ \ \ \ \ \ \ \ \ \ \  \large{  Keimei Gakuin High School}


\begin{abstract}
In this study, we study the relation between Grundy numbers of a Maximum Nim and Josephus problem.
Let  $ f(x) = \left\lfloor \frac{x}{k} \right\rfloor$, where $ \left\lfloor \ \  \right\rfloor$ is a floor function and $k$ is a positive integer.
We prove that there is a simple relation with a Maximum Nim with the rule function $f$ and the Josephus problem in which every $k$-th numbers are to be removed.
\end{abstract}

\section{Introduction}
Let $\mathbb{Z}_{\ge 0}$ and $\mathbb{N}$ represent sets of non-negative integers and positive integers, respectively.
For any real number $x$, the  floor of $x$ denoted by  $ \left\lfloor x \right\rfloor$ represent 
the least integer greater than or equal to $x$.

The classic game of Nim is played using stone piles. Players can remove any number of stones from any pile during their turn, and the player who removes the last stone is considered the winner. See \cite{bouton}.
 
 Several variants of the classical Nim game exist. For the Maximum Nim, we place an upper bound $f(n)$ on the number of stones that can be removed in terms of the number $n$ of stones in the pile (see \cite{levinenim}). For other articles on Maximum Nim, see \cite{levinenim}, \cite{thaij2023b} and  \cite{integer2023}.

 In this article, we study the relation between a Maximum Nim and  the Josephus problem. 
 
 Let  $ f(x) = \left\lfloor \frac{x}{k} \right\rfloor$, where $ \left\lfloor \ \  \right\rfloor$ is a floor function and $k$ is a positive integer.
We prove that there is a simple relation with a Maximum Nim with the rule function $f$ and the Josephus problem in which every $k$-th numbers are to be removed.
 This is remarkable because the games for Nim and Josephus' problems are thought to be entirely different.

\section{Maximum Nim}
We consider maximum Nim as follows.

Suppose there is a pile of $n$ stones, and two players take turns to remove stones from the pile.
At each turn, the player is allowed to remove at least one and at most 
$f(m)$ stones if the number of stones is $m$. The player who removes the last stone or stones is the winner. 
Here, $f(m)$ represents a function whose values are non-negative integers for $m \in Z_{\ge 0}$ such that 
$0 \leq f(m) -f(m-1) \leq 1$ for any natural number $m$. We refer to $f$ as the rule function.

We briefly review some necessary concepts in combinatorial game theory (see \cite{lesson} for more details). 

We deal with impartial games with no draws, and therefore, there are only two outcome classes. \\
$(a)$ A position is called a $\mathcal{P}$-\textit{position} if it is a winning position for the previous player (the player who just moved), as long as he/she plays correctly at every stage.\\
$(b)$ A position is called an $\mathcal{N}$-\textit{position} if it is a winning position for the next player, as long as he/she plays correctly at every stage.

The Grundy number is one of the most important tools in research on combinatorial game theory, and we need ``move’’ for the definition of the Grundy number.\\

\begin{definition}\label{defofmexgrundy}
$(i)$ For any position $(x)$ in this game, the set of positions can be reached by precisely one move, denoted as \textit{move}$(x)$. \\	
$(ii)$ The \textit{minimum excluded value} $(\textit{mex})$ of a set $S$ of non-negative integers is the smallest non-negative integer that is not in S. \\
$(iii)$ Let $(x,y)$ be a position in the game. The associated \textit{Grundy number} is denoted by $\mathcal{G}(x)$ and is recursively defined as follows:
$\mathcal{G}(x) = \textit{mex}(\{\mathcal{G}(u): (u) \in move(x)\}).$
\end{definition}

For the maximum Nim of $x$ stones with rule function $f(x)$, \\
$\textit{move}(x)$ $= \{x-u:1 \leq u \leq f(x) \text{ and } u \in N \}$.

The next result demonstrates the usefulness of the Sprague-Grundy theory in impartial games.
\begin{theorem}\label{theoremoforgrundyn}
For any position $(x)$,
	$\mathcal{G}(x)=0$ if and only if $(x)$ is the $\mathcal{P}$-position.
\end{theorem}
See \cite{lesson} for the proof of this theorem.\\

\begin{lemma}\label{lemmabylevinenim}
Let $\mathcal{G}$ represent the Grundy number of the maximum Nim with the rule function $f(x)$. Then, we have the following properties:\\
$(i)$ If $f(x) = f(x-1)$, $\mathcal{G}(x) = \mathcal{G}(x-f(x)-1)$.\\
$(ii)$ If $f(x) > f(x-1)$, then  $\mathcal{G}(x) = f(x)$.\\
\end{lemma}
These properties are proved in Lemma 2.1 of \cite{levinenim}.

\section{Maximum Nim with  the rule function $f(x)=\left\lfloor \frac{x}{k} \right\rfloor$ for $k \in \mathbb{N}$ and Josephus Problem. }
Let $k \in \mathbb{N}$ such that $k \geq 2$.
Throughout the remainder of this article we assume that  $\mathcal{G}$ is the Grundy number of Maximum Nim with  the rule function $f(x)=\left\lfloor \frac{x}{k} \right\rfloor$.
We study Josephus problem and its relation to Maximum Nim treated in previous sections.
\begin{lemma}\label{lemmaformax}
(i) When $x$ is a multiple of $k$,  $\mathcal{G}(x) =  \frac{x}{k} $.\\
(ii) When $x$ is not a multiple of $k$, $\mathcal{G}(x) = \mathcal{G}(\left\lfloor \frac{(k-1)x}{k} \right\rfloor)$.\\
\end{lemma}
\begin{proof}
Since $f(x) > f(x-1)$ for $x$ that is a multiple of $k$ and 
\begin{equation}
x-f(x) -1 = x-\left\lfloor \frac{x}{k} \right\rfloor -1 = \left\lfloor \frac{(k-1)x}{k} \right\rfloor,
\end{equation}
(i) and (ii) of this Lemma is deprived from (i) and (ii) of Lemma \ref{lemmabylevinenim}.
\end{proof}

\begin{definition}\label{defofjj}
We have a finite sequence of positive integers $1,2,3$
$, \cdots, n-1, n$ arranged in a circle, and we start to remove every $k$-th number until only one remains.
This is a well-known Josephus problem.
For $m$ such that $1 \leq m \leq n$ and $0 \leq i \leq n-1$, we define
\begin{equation}
JJ_k(n,m)=n-i
\end{equation}
when $m$ is the $i$-th number to be removed.
\end{definition}

\begin{example}
 Let $n=10$ and $k=3$. We start with numbers $\{1,2,3,4,5,6,7,8,9,10\}$, and remove every $k$-th number.
 In the first round, we remove $3,6,9$, and in the second round,
 we remove $2,7$. In the third round, we remove $1,8$, and in the fourth round,
 we remove $5$. In the fifth round, we do not remove any, and in the sixth round we remove $10$. The remaining number is $4$. Therefore, the numbers to be removed are $\{3, 6, 9, 2, 7, 1, 8, 5, 10\}$ in this order.
 Therefore, by Definition \ref{defofjj}, 
$JJ_3(10,3)=10-1=9$, $JJ_3(10,6)=10-2=8, \dots, $ $JJ_3(10,5)=10-8=2,$ $JJ_3(10,10)=10-9=1$.

In Theorem \ref{grundytheoremjj}, we use the word "the first stage" and "the second stage".
In this example, when we remove $3,6,9$, we denote this as the end of the first stage. Then,
we begin with the remaining number $10,1,2,4,5,7,8$, and we denote this as the second stage.
\end{example}

\begin{figure}[H]
\begin{tabular}{cc}
\begin{minipage}[t]{0.4\textwidth}
\begin{center}
\includegraphics[height=5cm]{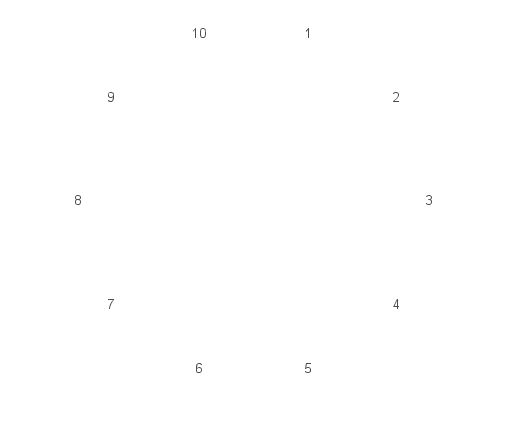}
\caption{the start}
\label{jose0}
\end{center}
\end{minipage}
\begin{minipage}[t]{0.4\textwidth}
\begin{center}
\includegraphics[height=5cm]{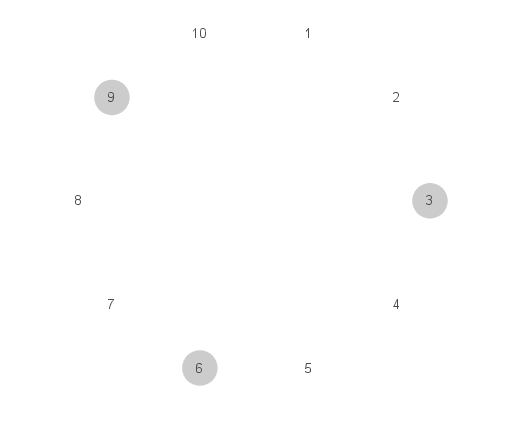}
\caption{the first Round}
\label{jose1}
\end{center}
\end{minipage}
\end{tabular}
\end{figure}

\begin{figure}[H]
\begin{tabular}{cc}
\begin{minipage}[t]{0.4\textwidth}
\begin{center}
\includegraphics[height=5cm]{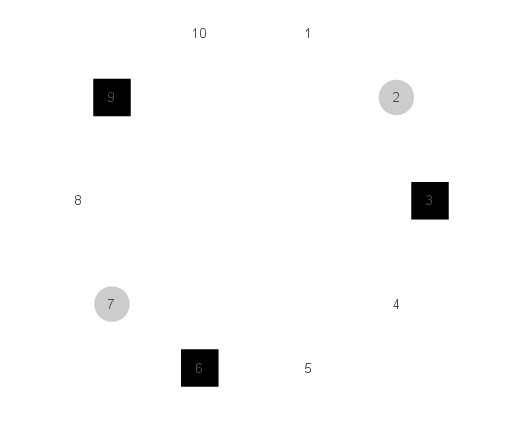}
\caption{the second Round}
\label{jose2}
\end{center}
\end{minipage}
\begin{minipage}[t]{0.4\textwidth}
\begin{center}
\includegraphics[height=5cm]{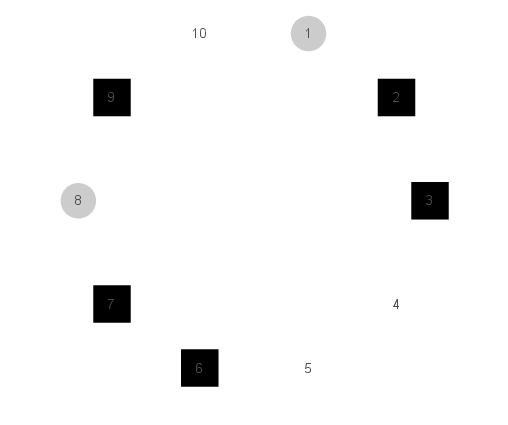}
\caption{the third Round}
\label{jose3}
\end{center}
\end{minipage}
\end{tabular}
\end{figure}

\begin{figure}[H]
\begin{tabular}{cc}
\begin{minipage}[t]{0.4\textwidth}
\begin{center}
\includegraphics[height=5cm]{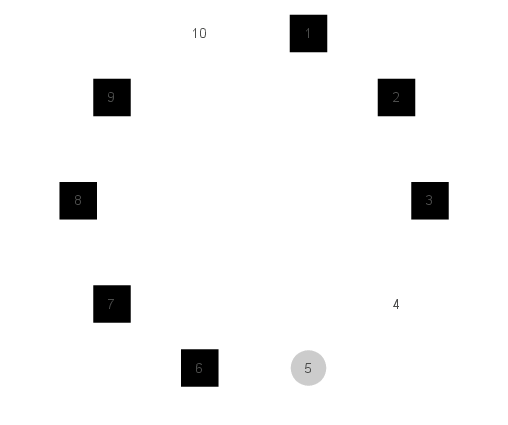}
\caption{the fourth Round}
\label{jose4}
\end{center}
\end{minipage}
\begin{minipage}[t]{0.4\textwidth}
\begin{center}
\includegraphics[height=5cm]{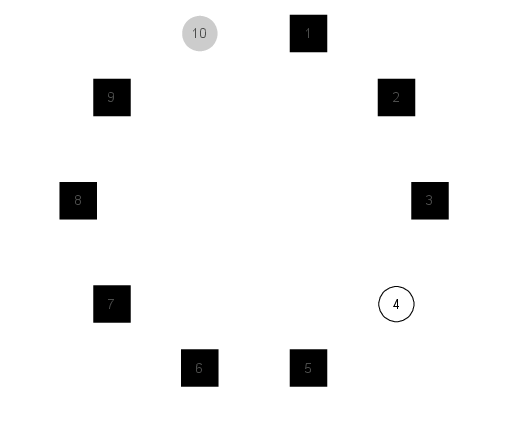}
\caption{the sixth Round}
\label{jose6}
\end{center}
\end{minipage}
\end{tabular}
\end{figure}

\begin{theorem}\label{grundytheoremjj}
For any $m \in \mathbb{N}$ such that $1 \leq m \leq n$,
\begin{equation}
JJ_k(n,m)=\mathcal{G}(nk-m) \label{theoremforgj}
\end{equation}
\end{theorem}
\begin{proof}
We prove by mathematical induction, and assume that 
\begin{equation}
JJ_k(n^{\prime},m)=\mathcal{G}(n^{\prime}k-m) \label{theoremforgj2}
\end{equation}
for $n^{\prime}$ such that $n^{\prime} < n$.
We compare $n$ positive integers 
\begin{equation}
\{1,2,3,\dots, k, \dots , n \}\label{josen}
\end{equation}
and $n$  Grundy numbers 
\begin{equation}
\{\mathcal{G}(nk-1), \mathcal{G}(nk-2), \mathcal{G}(nk-3), \dots, \mathcal{G}(nk-n) \}.\label{grundyn}
\end{equation}

We arrange  $n$ numbers of (\ref{josen}) in a circle and begin to remove every $k$-th number. At the same time, we calculate the values of every $k$-th Grundy number as the following.

We start the removing process of Josephus Problem with the numbers in (\ref{josen}).
First, we choose $k$ in (\ref{josen}), and remove it. At the same time, we choose the $k$-th number 
in (\ref{grundyn}) that is $\mathcal{G}(nk-k)$. By (i) of  Lemma \ref{lemmaformax}, 
$\mathcal{G}(nk-k)=n-1$. This is the first number to be removed in this Josephus Problem.
By Definition \ref{defofjj}, 
\begin{equation}
JJ_k(n,k)=n-1=\mathcal{G}(nk-k).
\end{equation}
Therefore, Equation (\ref{theoremforgj}) is valid for $m=k$.

Second, we choose the $2k$ in (\ref{josen}) and remove it. At the same time, we choose $2k$-th number in (\ref{grundyn}) that is $\mathcal{G}(nk-2k)$. By (i) of  Lemma \ref{lemmaformax}, 
$\mathcal{G}(nk-2k)=n-2$. 
By Definition \ref{defofjj}, 
\begin{equation}
JJ_k(n,2k)=n-2=\mathcal{G}(nk-2k).
\end{equation}
Therefore, Equation (\ref{theoremforgj}) is valid for $m=2k$.
We continue this until the end of the first round.
For $n$, there exist $t,u \in \mathbb{Z}_{\ge 0}$ such that $0 \leq u \leq k-1$ and  
\begin{equation}
n=kt+u. \label{conditionforlastn}
\end{equation}

We prove the case that $u=1$ as other cases can be proved similarly.
Then, by (\ref{conditionforlastn}), we have $n-1=kt$. Therefore,
in the Josephus Problem of (\ref{josen}), $n-1$ is $t$-th number to be removed
in the first round, and this is the last number to be removed in the first round.

Correspondingly, the $t$-th Grundy number in (\ref{grundyn}) is 
$\mathcal{G}(nk-(n-1))=\mathcal{G}(nk-kt)=n-t$.
By Definition \ref{defofjj}, we have 
\begin{equation}
JJ_k(n,n-1)=n-t=\mathcal{G}(nk-(n-1)).
\end{equation}
Therefore, Equation (\ref{theoremforgj}) is valid for $m=n-1$.

Next, we start removing the process with the last number in 
\begin{equation}
\{1,2,3,\dots,k-1, k+1, \dots ,2k-1, 2k+1,  \dots ,n-2, n \}.\nonumber
\end{equation}
Then, we start to remove numbers with 
\begin{equation}
\{n, 1,2,3,\dots,k-1, k+1, \dots ,2k-1, 2k+1,  \dots ,n-2 \}\label{josen2}
\end{equation}
as the second stage of the Josephus problem.

Then, we start to remove the remaining Grundy numbers 
\begin{align}
& \mathcal{G}(nk-n). \mathcal{G}(nk-1),\mathcal{G}(nk-2),  \dots, \mathcal{G}(nk-k+1), \nonumber \\
& \mathcal{G}(nk-k-1),\dots, \mathcal{G}(nk-2k+1),\mathcal{G}(nk-2k-1)  \nonumber \\
& , \dots, \mathcal{G}(nk-n+2)\label{remaingrundy}
\end{align}
as the second stage of the removing Grundy numbers.

By (ii) of Lemma \ref{lemmaformax}, 
\begin{align}
\mathcal{G}(nk-n)& = \mathcal{G}((n-t)k-1),\nonumber \\
\mathcal{G}(nk-1)& =\mathcal{G}(\lfloor \frac{nk(k-1)-(k-1)}{k}\rfloor)\nonumber \\
& = \mathcal{G}(n(k-1)-1)\nonumber \\
& = \mathcal{G}((n-t)k-2),\nonumber \\
\mathcal{G}(nk-2)& =\mathcal{G}(\lfloor \frac{nk(k-1)-2(k-1)}{k}\rfloor)\nonumber \\
& = \mathcal{G}(n(k-1)-2)\nonumber \\
& = \mathcal{G}((n-t)k-3),\nonumber \\
\vdots \nonumber \\
\mathcal{G}(nk-(k-1))& =\mathcal{G}(\lfloor \frac{(k-1)(nk-(k-1))}{k}\rfloor)\nonumber \\
& =\mathcal{G}(\lfloor \frac{(k-1)nk  -(k-1)k +k-1}{k}\rfloor)\nonumber \\
& = \mathcal{G}((k-1)n-k+1)  \nonumber \\
& = \mathcal{G}((n-t)k-k),  \nonumber \\
\mathcal{G}(nk-(k+1))& =\mathcal{G}(\lfloor \frac{(k-1)(nk-(k+1))}{k}\rfloor)\nonumber \\
& =\mathcal{G}(\lfloor \frac{(k-1)nk  -(k-1)k -k+1}{k}\rfloor)\nonumber \\
& = \mathcal{G}((k-1)n-k)  \nonumber \\
& = \mathcal{G}((n-t)k-k-1)  \nonumber \\
&  \vdots  \nonumber \\
\mathcal{G}(nk-(n-2))& =\mathcal{G}(\lfloor \frac{nk(k-1)-(n-2)(k-1)}{k}\rfloor)\nonumber \\
& = \mathcal{G}(n(k-1)-n+t+1)\nonumber \\
& = \mathcal{G}((n-t)k-n+t).\nonumber 
\end{align}
By the above, we have the following $n-t$ Grundy numbers.
\begin{equation}
\mathcal{G}((n-t)k-1), \mathcal{G}((n-t)k-2), \dots, \mathcal{G}((n-t)k-n+t).\label{secondroundjo}
\end{equation}

Then, we begin to remove numbers of (\ref{secondroundjo}) in the second stage of Josephus problem. Here we use the hypothesis of mathematical induction, and 
by (\ref{theoremforgj2}), 
if 
\begin{equation}
\mathcal{G}((n-t)k-1)=i,
\end{equation}
then this first number in (\ref{secondroundjo}) is the $(n-t)-i$-th removed number in this second stage of Josephus Problem. In the first stage of the Josephus Problem, we have already removed $t$ numbers, and hence $\mathcal{G}((n-t)k-1)$ is $n-i$-th removed Grundy number when we count the order from the start of the first round.
Then, we have $\mathcal{G}((nk-n)=\mathcal{G}((n-t)k-1)=i$ and $\mathcal{G}((n-t)k-1)$ is $n-i$-th removed number. Therefore, we have
\begin{equation}
JJ_k(n,n)=i = \mathcal{G}(nk-n). \nonumber
\end{equation}
Therefore, (\ref{theoremforgj}) is valid for $m=n$.

Next, we prove (\ref{theoremforgj}) for arbitrary Grundy number 
in (\ref{secondroundjo}).

For $s,i \in \mathbb{Z}_{\ge 0}$ with $1 \leq i < k-1$, there exists $u \in \mathbb{N}$ such that 
\begin{align}
\mathcal{G}(nk-(sk+i)) & =\mathcal{G}(\lfloor \frac{nk(k-1)-sk(k-1)-i(k-1)}{k}\rfloor)\nonumber \\
& = \mathcal{G}((n-t)k-sk-i+s-1+\frac{i}{k})  \nonumber \\
& = \mathcal{G}((n-t)k-sk-i+s-1) = \nonumber \\
& = \mathcal{G}((n-t)k-(s(k-1)+i+1)) = u.  \label{grundyequ}
\end{align}
Here 
\begin{equation}
 \mathcal{G}(kn-(sk+i))=\mathcal{G}((n-t)k-(s(k-1)+i+1))
\end{equation}
is the $s(k-1)+i+1$-th number in (\ref{secondroundjo}).
Here we use the hypothesis of mathematical induction, and by (\ref{theoremforgj2}) and (\ref{grundyequ}), this 
$s(k-1)+i+1$-th number in (\ref{secondroundjo}) will be the 
$(n-t)-u$-th removed number in (\ref{secondroundjo}).
Since we have already removed $t$ numbers in the first stage of the Josephus problem,
$\mathcal{G}(kn-(sk+i))$ is the $(n-t)-u+t=n-u$-th removed number from the start of the Josephus problem.
Therefore, we have
\begin{equation}
 \mathcal{G}(kn-(sk+i))=u
\end{equation}
and
$\mathcal{G}(kn-(sk+1))$ is the $n-u$-th number,
and 
(\ref{theoremforgj}) is valid for $sk+1$.
\end{proof}

\bibliographystyle{amsplain}


\end{document}